\documentclass[10pt]{amsart}
\usepackage{amsmath}
\usepackage{amsfonts,amssymb,amsthm,amstext,amscd,boxedminipage}  
\usepackage{setspace}
\usepackage{fullpage}
\usepackage{graphicx}
\usepackage{graphicx}
\usepackage {psfrag} 
\usepackage{color}
\usepackage{tikz}
\usepackage{youngtab}
\newcommand{\blda}{{\bf 4}}
\newcommand{\bldb}{{\bf 2}}
\newcommand{\bldc}{{\bf 3}}

\newtheorem{lemma}{Lemma}

\newtheorem{theorem}{Theorem}
\newtheorem{prop}{Proposition}

\title{An explicit bijection between semistandard tableaux and non-elliptic $sl_3$ webs}

\author{Heather M. Russell}
\address{Department of Mathematics \\ University of Southern California}
\email{heathemr@usc.edu}

\begin{document}
\begin{abstract}
The $sl_3$ spider is a diagrammatic category used to study the representation theory of the quantum group $U_q(sl_3)$. The morphisms in this category are generated by a basis of non-elliptic webs. Khovanov-Kuperberg observed that non-elliptic webs are indexed by semistandard Young tableaux. They establish this bijection via a recursive growth algorithm. Recently, Tymoczko gave a simple version of this bijection in the case that the tableaux are standard and used it to study rotation and joins of webs. We build on Tymoczko's bijection to give a simple and explicit algorithm for constructing all non-elliptic $sl_3$ webs.
\end{abstract}
\maketitle

\section{Introduction}
The $sl_3$ spider, introduced by Kuperberg~\cite{K} and subsequently studied by many others~\cite{KK, Kim, SM, MOY}, is a diagrammatic, braided monoidal category encoding the representation theory of $U_q(sl_3)$. The objects in this category, called sign strings, are finite words in the alphabet $\{+,-\}$ including the empty word. The morphisms are $\mathbb{Z}[q, q^{-1}]$ - linear combinations of certain graphs called webs. See Figure \ref{webexample} for an example of a web.

The objects in the spider can be thought of as tensor products of the two dual 3-dimensional irreducible representations $V^{+}$ and $V^{-}$ of $U_q(sl_3)$, and the morphisms can be thought of as intertwining maps between tensor products of these representations~\cite{K}. Spider categories for other Lie types have been defined. See for instance~\cite{Fon, SM, MOY}. 

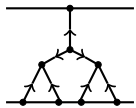
\begin{figure}[h]
\begin{tikzpicture}[baseline=0cm, scale=0.5]
\draw[style=thick](-1.7,0)--(1.7,0);
\draw[style=thick](-1.7,2.5)--(1.7,2.5);
\draw[style=thick, ->] (-1.25,0)--(-1,.5);
\draw[style=thick, -] (-1,.5)--(-.75,1);
\draw[radius=.08, fill=black](-.75,1)circle;
\draw[radius=.08, fill=black](0,2.5)circle;
\draw[radius=.08, fill=black](-1.25,0)circle;
\draw[radius=.08, fill=black](-.3,0)circle;
\draw[radius=.08, fill=black](.3,0)circle;
\draw[radius=.08, fill=black](1.25,0)circle;
\draw[style=thick, ->] (-.25,0)--(-.5,.5);
\draw[style=thick, -] (-.5,.5)--(-.75,1);

\draw[style=thick] (-.75,1)--(-.4,1.2);
\draw[style=thick,<-](-.4,1.2)--(0,1.4);

\draw[radius=.08, fill=black](0,1.4)circle;
\draw[style=thick,->](0,1.4)--(0,1.95);
\draw[style=thick](0,1.95)--(0,2.5);

\draw[ style=thick, ->] (1.25,0)--(1,.5);
\draw[style=thick, -] (1,.5)--(.75,1);
\draw[radius=.08, fill=black](.75,1)circle;
\draw[style=thick, ->] (.25,0)--(.5,.5);
\draw[style=thick, -] (.5,.5)--(.75,1);

\draw[style=thick] (.75,1)--(.4,1.2);
\draw[style=thick,<-](.4,1.2)--(0,1.4);
\end{tikzpicture}
\caption{A web in $\textup{Hom}(++++,+)$.}\label{webexample}
\end{figure}

Webs in the $sl_3$ spider are oriented trivalent graphs drawn in a rectangular region with boundary points lying on the top and bottom edges of that region. Edges incident on the boundary points have orientations compatible with the source and target sign strings. We read webs from bottom to top. All vertices are either sources or sinks. Webs are also subject to Relations \ref{circlerel}, \ref{bigonrel}, and \ref{squarerel} below which are often referred to as the circle, bigon, and square relations respectively. A web with no bigons, squares, or circles is called non-elliptic or irreducible. Every web is a linear combination of non-elliptic webs. We follow the normalization conventions found in Khovanov's work on $sl_3$ link homology~\cite{MK}.
\begin{equation} \label{circlerel}
\raisebox{1pt}{\begin{tikzpicture}[baseline=0cm, scale=0.5]
\draw[style=thick,->] (0,0) arc (0:180:1cm);
\draw[style=thick](-2,0) arc (180:360:1cm);
\end{tikzpicture}} = [3]_q = q^2 + 1 + q^{-2} 
 \end{equation}
 \begin{equation}\label{bigonrel}
 \raisebox{-5pt}{\begin{tikzpicture}[baseline=0cm, scale=0.5]
\draw[style=thick, -<] (0,1) -- (0,1.5);
\draw[style=thick](0,1.5)--(0,2);
\draw[style=thick](0,-1)--(0,-.5);
\draw[style=thick, ->](0,0)--(0,-.5);
\draw[style=thick,->](0,0) .. controls (-.5,.1) and (-.5,.35).. (-.5,.5);
\draw[style=thick](-.5,.5) .. controls (-.5,.65) and (-.5,.9).. (0,1);
\draw[style=thick,->](0,0) .. controls (.5,.1) and (.5,.35).. (.5,.5);
\draw[style=thick](.5,.5) .. controls (.5,.65) and (.5,.9).. (0,1);
\draw[radius=.08, fill=black](0,0)circle;
\draw[radius=.08, fill=black](0,1)circle;
\end{tikzpicture}}= [2]_q \; \;\raisebox{-33pt}{\begin{tikzpicture}[baseline=0cm, scale=0.5]
\draw[style=thick, -<] (0,1) -- (0,2.5);
\draw[style=thick](0,2.5)--(0,4);
\end{tikzpicture}} = (q+q^{-1}) \;\;  \raisebox{-33pt}{\begin{tikzpicture}[baseline=0cm, scale=0.5]
\draw[style=thick, -<] (0,1) -- (0,2.5);
\draw[style=thick](0,2.5)--(0,4);
\end{tikzpicture}}
 \end{equation}
 \begin{equation}\label{squarerel}
 \raisebox{-18pt}{\begin{tikzpicture}[baseline=0cm, scale=0.5]
\draw[style=thick, -<] (0,1) -- (0,1.5);
\draw[style=thick](0,1.5)--(0,2);
\draw[style=thick,->](0,2)--(.5,2);
\draw[style=thick](.5,2)--(1,2);
\draw[style=thick, ->](1,1)--(1,1.5);
\draw[style=thick](1,1.5)--(1,2);
\draw[style=thick, ->](1,1)--(.5,1);
\draw[style=thick](.5,1)--(0,1);
\draw[style=thick,-<](1,2) -- (1.25,2.25);
\draw[style=thick](1.25,2.25)--(1.5,2.5);
\draw[style=thick,-<](0,1) -- (-.25, .75);
\draw[style=thick](-.25, .75)--(-.5, .5);
\draw[style=thick,->](1,1) -- (1.25,.75);
\draw[style=thick](1.25,.75)--(1.5,.5);
\draw[style=thick,->](0,2) -- (-.25,2.25);
\draw[style=thick](-.25,2.25)--(-.5,2.5);
\draw[radius=.08, fill=black](0,1)circle;
\draw[radius=.08, fill=black](0,2)circle;
\draw[radius=.08, fill=black](1,2)circle;
\draw[radius=.08, fill=black](1,1)circle;

\end{tikzpicture}} = \raisebox{-10pt}{\begin{tikzpicture}[baseline=0cm, scale=0.9]
\draw[xshift=1.5cm, style=thick,-<](-.25,0) .. controls (-.5,.2) and (-.5,.5).. (-.5,.5);
\draw[xshift=1.5cm, style=thick](-.5,.5) .. controls (-.5,.45) and (-.5,.9).. (-.25,1);
\draw[style=thick,->](.25,0) .. controls (.5,.2) and (.5,.5).. (.5,.5);
\draw[style=thick](.5,.5) .. controls (.5,.45) and (.5,.9).. (.25,1);
\end{tikzpicture}} +  \raisebox{-16pt}{\begin{tikzpicture}[baseline=0cm, scale=0.9, rotate=90]
\draw[xshift=1.5cm, style=thick,](-.25,0) .. controls (-.5,.2) and (-.5,.5).. (-.5,.5);
\draw[xshift=1.5cm, style=thick](-.5,.5) .. controls (-.5,.45) and (-.5,.9).. (-.25,1);
\draw[xshift=1.5cm,style=thick](-.5,.5)--(-.6,.4);
\draw[xshift=1.5cm,style=thick](-.5,.5)--(-.4,.4);

\draw[style=thick,](.25,0) .. controls (.5,.2) and (.5,.5).. (.5,.5);
\draw[style=thick](.5,.5) .. controls (.5,.45) and (.5,.9).. (.25,1);
\draw[style=thick](.5,.5)--(.6,.6);
\draw[style=thick](.5,.5)--(.4,.6);
\end{tikzpicture}}
\end{equation}

Given a sign string $s$, construct the dual string $s^*$ of $s$ by reversing the order of $s$ and then replacing each $+$ with a $-$ and each $-$ with a $+$. This is really just a diagrammatic version of the statement that, for quantum group representations,  $(V\otimes W)^* \cong W^* \otimes V^*$. Let $\textup{Inv}(V)$ be the space of invariant tensors of $V$ where $V$ is a tensor product of irreducible representations of $U_q(sl_3)$. Since $\textup{Hom}(V,W) \cong \textup{Inv}(V^*, W)$, it is enough to study $sl_3$ webs of the form $\textup{Hom}(s, \emptyset)$.

Let $s=s_1 \ldots s_n$ be a sign string. The dimension of $\textup{Inv}(V^{s_1}\otimes \cdots \otimes V^{s_n})$ is the number of lattice paths in the dominant  Weyl chamber from the origin to itself satisfying some additional condition coming from the string~\cite{KK}. These dominant lattice paths for $s$ correspond to certain words in the alphabet $\{-1, 0, +1\}$. Khovanov-Kuperberg give a recursive growth algorithm which produces a non-elliptic web from a given lattice path word. This growth algorithm establishes a bijection between dominant lattice paths and webs with inverse coming from a depth map on webs~\cite{KK}.

Recall that a semistandard Young tableau is a filling of a Young diagram which strictly increases in columns and weakly increases in rows. The dimension of the invariant tensor space is reformulated by Petersen-Pylyavskyy-Rhoades using the language of semistandard tableaux~\cite{PPR}.
\begin{prop}
Let $s$ be a sign string of length $3n$ with $k$ minuses and $3n-k$ pluses. The number of non-elliptic webs in $Hom(s, \emptyset)$ is equal to the number of semistandard tableaux of shape $(3, 3, \ldots, 3, 3) \vdash  3n$ filled with $\{1^2, \ldots, k^2, k+1, \ldots, 3n-k\}$.
\end{prop}

Tymoczko recently gave an explicit bijection between webs in $Hom(+++ \ldots +++, \emptyset)$ and standard tableaux~\cite{T}. This is accomplished by constructing an intermediate object called an $m$-diagram which can then be modified slightly to produce a non-elliptic web. Tymoczko shows that this straightforward procedure provides a concrete realization of the growth algorithm bijection of Khovanov-Kuperberg. 

Building on the $m$-diagram algorithm, this paper provides a simple bijection between all non-elliptic webs and a certain subset of semistandard Young tableaux. We begin by recalling Tymoczko's $m$-diagram algorithm and then describe the generalized bijection providing many examples.  We conclude with two theorems about rotation and join of webs that generalize results of Petersen-Pylyavskyy-Rhoades and Tymoczko  to all $sl_3$ webs~\cite{PPR, T}.

An interesting potential application of this bijection is in the study of Spaltenstein varieties. Combinatorial data from $sl_2$ webs has been used to describe the representation theory and topological structure of Springer varieties, certain flag varietyies used to construct irreducible representations of the symmetric group~\cite{FKK, Fung, KhCr, HMR, Russ, RT, SW}. Spaltenstein varieties are a generalization of Springer varieties using partial flags. 

Just as the components of Springer varieties are indexed by standard tableaux, the components of Spaltenstein varieties are indexed by semistandard tableaux. Recent work of Brundan-Ostrik and Sch\"{a}fer show strong evidence that the combinatorics of the more general class of webs studied here should aid in the study of three-row Spaltenstein varieties~\cite{BO, Sch}.

\subsubsection*{Acknowledgements}
We are very grateful to Julianna Tymoczko for suggesting this project and for many enlightening conversations. Thanks also to Matt Housley for helpful discussions about $\tau$ sets. We also wanted to acknowledge Dongho Moon who has recently obtained similar results.

\section{Tymoczko's m-diagram algorithm}\label{JTalg}

Let $n\in \mathbb{N}$ and consider the partition $(n,n,n) \vdash 3n$. Let $T$ be a standard tableau of shape $(n,n,n)$.  The bijection in this section is between tableaux $T$ and webs with $3n$ source vertices. Note that the number of standard fillings of shape $(n,n,n)$ is the same as the number of standard fillings of the shape $(3, 3, \ldots, 3, 3)$, so this can also be thought of as a bijection  with tableaux of that shape. 

Given a tableau $T$,  the Tymoczko $m$-diagram algorithm constructs the $m$-diagram $m_T$ as follows~\cite{T}.
\begin{itemize}
\item{Draw a horizontal line with $3n$ equally spaced dots labeled from left to right with the numbers $1, \ldots, 3n$. This line forms the lower boundary for the diagram, and all arcs will lie above this line.}
\item{Starting with the smallest number $j$ on the second row, draw a semi-circular arc connecting $j$ to its nearest unoccupied neighbor $i$ to the left. The arcs $(i,j)$  are the left arcs in the $m$-diagram.}
\item{Starting with the smallest number $k$ on the bottom row, draw a semi-circular arc connecting $k$ to its nearest neighbor $j$ to the left that does not already have an arc coming to it from the left. The arcs $(j,k)$ are the right arcs of the $m$-diagram.}
\end{itemize}
The collection of left arcs is nonintersecting as is the collection of right arcs, but left arcs can intersect right arcs. Figure \ref{mdia} has an example of an $m$-diagram.

\begin{figure}[h]
\begin{tikzpicture}[baseline=0cm, scale=.7]
\node at (0,0) {\Large\young(13,25,46)};
\draw[style= ultra thick, ->] (1.5,0)--(2.25,0);
\draw[style=thick] (3,-.7)--(9,-.7);
\draw[radius=.08, fill=black](3.5,-.7)circle;
\draw[radius=.08, fill=black](4.5,-.7)circle;
\draw[radius=.08, fill=black](5.5,-.7)circle;
\draw[radius=.08, fill=black](6.5,-.7)circle;
\draw[radius=.08, fill=black](7.5,-.7)circle;
\draw[radius=.08, fill=black](8.5,-.7)circle;
\draw (4.5,-.7) arc (0:180: .5cm);
\draw (7.5,-.7) arc (0:180: 1cm);
\draw (6.5,-.7) arc (0:180: 1cm);
\draw (8.5,-.7) arc (0:180: .5cm);
\node at (3.5,-1.1) {1};
\node at (4.5,-1.1) {2};
\node at (5.5,-1.1) {3};
\node at (6.5,-1.1) {4};
\node at (7.5,-1.1) {5};
\node at (8.5,-1.1) {6};
\end{tikzpicture}
\caption{The $m$-diagram for a tableau.}\label{mdia}
\end{figure}
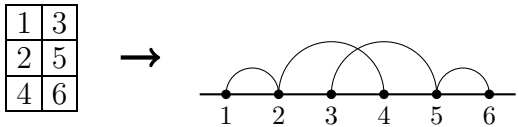

From an $m$-diagram $m_T$ for $T$ there is a straightforward procedure for transforming $m_T$ into a non-elliptic web $w_T$~\cite{T}. 
\begin{itemize}
\item{At each boundary vertex where two semi-circular arcs meet, replace the portion of the diagram in a small neighborhood of the vertex with a `Y` shape as shown in Figure \ref{middle}.}
\item{Orient all arcs away from the boundary so that the branching point of each `Y` becomes a source.}
\item{Finally replace any 4-valent intersection point of a left arc and a right arc with a pair of trivalent vertices as shown in Figure \ref{replace4}. There is a unique way to do this preserving orientation of incoming arcs.}
\end{itemize}

\begin{figure}[h]
\begin{tikzpicture}[baseline=0cm, scale=.6]
\draw[style=thick] (0,0)--(3,0);
\draw[radius=.08, fill=black](1.5,0)circle;
\draw [style=thick](1.5,0) arc (0:100: 1.2cm);
\draw [style=thick] (1.5,0) arc (180:80: 1.2cm);

\draw[style=ultra thick, ->] (3.25,.5) -- (4,.5);

\draw[style=thick] (4.25,0)--(7.25,0);
\draw[radius=.08, fill=black](5.75,0)circle;
\draw[radius=.08, fill=black](5.75,1)circle;
\draw[style=thick,->](5.75,0) -- (5.75,.5);
\draw[style=thick](5.75,.5) -- (5.75,1);
\draw[style=thick, -<] (5.75,1) -- (5.25,1.1);
\draw[style=thick](5.25,1.1) -- (4.75,1.2);
\draw[style=thick, -<] (5.75,1) -- (6.25,1.1);
\draw[style=thick](6.25,1.1) -- (6.75,1.2);
\end{tikzpicture}
\caption{Modifying the middle vertex of an $m$.} \label{middle}
\end{figure}
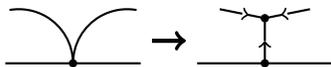

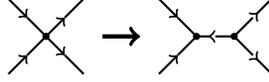
\begin{figure}[h]
\begin{tikzpicture}[baseline=0cm, scale=0.5]
\draw[style=thick, ->] (-1,-1)--(-.5,-.5);
\draw[style=thick, ->] (-1,1)--(-.5,.5);
\draw[style=thick, ->] (-.5,.5)--(.5,-.5);
\draw[style=thick, ->] (-.5,-.5)--(.5,.5);
\draw[style=thick] (.5,.5)--(1,1);
\draw[style=thick] (.5,-.5)--(1,-1);
\draw[radius=.08, fill=black]circle;

\draw[style=ultra thick, ->] (1.5,0)--(2.5,0);

\draw[xshift=4cm, style=thick, ->] (-1,-1)--(-.5,-.5);
\draw[xshift=4cm, style=thick, ->] (-1,1)--(-.5,.5);
\draw[xshift=4cm, style=thick] (-.5,.5)--(0,0);
\draw[xshift=4cm, style=thick ] (-.5,-.5)--(0,0);
\draw[xshift=4cm, style=thick, -< ] (0,0)--(.5,0);
\draw[xshift=4cm, style=thick ] (.5,0)--(1,0);
\draw[xshift=4.5cm, style=thick,>-] (1,.5)--(1.5,1);
\draw[xshift=4.5cm, style=thick] (.5,0)--(1,.5);
\draw[xshift=4.5cm, style=thick] (.5,0)--(1,-.5);
\draw[xshift=4.5cm, style=thick, >-] (1,-.5)--(1.5,-1);
\draw[xshift=4cm, radius=.08, fill=black]circle;
\draw[xshift=5cm, radius=.08, fill=black]circle;
\end{tikzpicture}
\caption{Replacing a 4-valent vertex with trivalent vertices.}\label{replace4}
\end{figure}

For each face of a web $w$, define its depth to be the minimal number of times a path from the given face to the unbounded region must intersect $w$. An example is shown in Figure \ref{depthex}. Depths of adjacent faces differ by at most one. 

\begin{figure}[h]
\begin{tikzpicture}[baseline=0cm, scale=0.7]
\draw[style=thick] (2.5,0)--(8.5,0);
\draw[radius=.08, fill=black](3,0)circle;
\draw[radius=.08, fill=black](4,0)circle;
\draw[radius=.08, fill=black](5,0)circle;
\draw[radius=.08, fill=black](6,0)circle;
\draw[radius=.08, fill=black](7,0)circle;
\draw[radius=.08, fill=black](8,0)circle;
\draw[style=thick,->](4, 0) -- (4,.5);
\draw[style=thick](4,.5)--(4,1);
\draw[radius=.08, fill=black](4,1)circle;
\draw[style=thick,->](3,0)--(3.5,.5);
\draw[style=thick](3.5,.5)--(4,1);
\draw[style=thick,->](7, 0) -- (7,.5);
\draw[style=thick](7,.5)--(7,1);
\draw[radius=.08, fill=black](7,1)circle;
\draw[style=thick,->](8,0)--(7.5,.5);
\draw[style=thick](7.5,.5)--(7,1);
\draw[radius=.08, fill=black](5.5,1)circle;
\draw[style=thick,-<](5.5,1)--(5.5,1.5);
\draw[style=thick](5.5,1.5)--(5.5,2);
\draw[radius=.08, fill=black](5.5,2)circle;
\draw[style=thick,->](5,0)--(5.25,.5);
\draw[style=thick](5.25,.5)--(5.5,1);
\draw[style=thick,->](6,0)--(5.75,.5);
\draw[style=thick](5.75,.5)--(5.5,1);
\draw[style=thick,-<](4,1)--(4.75,1.5);
\draw[style=thick](4.75,1.5)--(5.5,2);
\draw[style=thick,-<](7,1)--(6.25,1.5);
\draw[style=thick](6.25,1.5)--(5.5,2);
\node at (2.7,.3) {0};
\node at (3.7,.3) {1};
\node at (4.7,.3) {1};
\node at (5.5,.3) {2};
\node at (6.3,.3) {1};
\node at (7.3,.3) {1};
\node at (8.3, .3) {0};
\end{tikzpicture}
\caption{The depth map for a web}\label{depthex}
\end{figure}
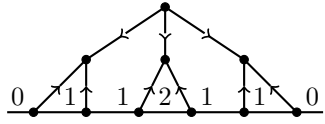

Let $F_{i,L}$ be the face immediately to the left of the edge incident on boundary vertex $i$, and let $F_{i,R}$ be the face immediately to the right of the edge incident on $i$. The following algorithm constructs a standard tableaux $T_w$ from a non-elliptic web $w$ with boundary $3n$ sources. In fact this process is inverse to Tymoczko's web bijection in the sense that $T_{w_T} = T$~\cite{T}.
\begin{itemize}
\item{If the depth of $F_{i,L}$ is less than the depth of $F_{i,R}$, put $i$ in the top row of $T_w$.} 
\item{If the depths of $F_{i,L}$ and $F_{i,R}$ are the same, put $i$ in the middle row of $T_w$.} \item{If the depth of $F_{i,L}$ is greater than the depth of $F_{i,R}$, put $i$ in the bottom row of $T_w$.}
\end{itemize}

Let $\tau(T)$ be the set of all pairs $(i, i+1)$ such that $i$ occurs in a row above $i+1$ in $T$. The terminology of $\tau$ comes from the work of Vogan on primitive spectra of semi simple Lie algebras ~\cite{V}. The $\tau$ set is also often called the descent set of a tableau. Lemma \ref{tausetlemma} appears in an upcoming paper of the author with Housley and Tymoczko where we study the symmetric group action on $sl_3$ webs with $3n$ sources~\cite{HRT}. It is the key idea in establishing a bijection between semistandard tableaux and webs.

\begin{lemma} \label{tausetlemma}
Given a standard tableau $T$ and its associated web $w_T$, if $(i,i+1)\in\tau(T)$ then boundary vertices $i$ and $i+1$ are connected to the same internal vertex in $w_T$.
\end{lemma}

\begin{proof}
Say $(i,i+1)\in \tau(T)$. Since $T$ has three rows, there are three possibilities:
\begin{enumerate}
\item{$i$ is in the top row, and $i+1$ is in the middle row.}
\item{$i$ is in the middle row, and $i+1$ is in the bottom row.}
\item{$i$ is in the top row, and $i+1$ is in the bottom row.}

\end{enumerate}

In the first case, the boundary vertices $i$ and $i+1$ must be connected by the left arc of an $m$. If this were not the case, then the $m$-diagram for $w_T$ would have two left arcs crossing, which cannot happen. Since $i$ and $i+1$ are adjacent and connected by the arc of an $m$, they will be connected to the same internal vertex in $w_T$.

The second case is completely analogous to the first except that $i$ and $i+1$ are connected by the right arc of an $m$. This once again means that they connect to the same internal vertex in $w_T$.

In the third case $i$ is at the far left of an $m$, and $i+1$ is at the far right of an $m$. Since there are no external vertices between them, they must cross exactly in the manner shown in Figure \ref{tauproof1}. This means that in $w_T$, vertices $i$ and $i+1$ will connect to the same internal vertex.

\begin{figure}[h]
\begin{tikzpicture}[baseline=0cm, scale=0.7]
\draw[style=thick](0,0)--(3,0);
\draw[radius=.08, fill=black](1,0)circle;
\draw[radius=.08, fill=black](2,0)circle;
\draw (1,0) arc (180:80: 1.2cm);
\draw (2,0) arc (0:100: 1.2cm);
\node at (1,-.3) {$i$};
\node at (2,-.3) {$i+1$};

\draw[style=ultra thick, ->] (3.5,.5)--(4.5,.5);

\draw[style=thick](5,0)--(8,0);
\draw[radius=.08, fill=black](6,0)circle;
\draw[radius=.08, fill=black](7,0)circle;
\draw[radius=.08, fill=black](6.5,.5)circle;
\draw[radius=.08, fill=black](6.5,1)circle;
\node at (6,-.3) {$i$};
\node at (7,-.3) {$i+1$};
\draw[style=thick,->](6,0)--(6.25,.25);
\draw[style=thick](6.25,.25)--(6.5, .5);
\draw[style=thick,->](7,0)--(6.75,.25);
\draw[style=thick](6.75,.25)--(6.5, .5);
\draw[style=thick,-<](6.5,.5)--(6.5,.75);
\draw[style=thick](6.5,.75)--(6.5,1);
\draw[style=thick,->](6.5,1)--(6.25,1.1);
\draw[style=thick](6.25,1.1)--(6,1.2);
\draw[style=thick,->](6.5,1)--(6.75,1.1);
\draw[style=thick](6.75,1.1)--(7,1.2);

\end{tikzpicture}
\caption{Case 3: Vertices $i$ and $i+1$ in the $m$-diagram and web.}\label{tauproof1}
\end{figure}
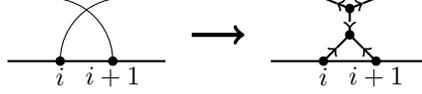
\end{proof}

\section{A bijection between semistandard tableaux and webs}

Given a semistandard tableau $T$, the content $\lambda$ of $T$ is the composition where $\lambda_i$ is the number of occurrences of $i$ in $T$.  Given a sign string $s=s_1 \cdots s_n$ with $k$ minuses define the content of $s$ to be the composition $\lambda_s = (\lambda_{s,1}, \ldots , \lambda_{s,{3n-k}})$ where 
\begin{displaymath}
   \lambda_{s,i} = \left\{
     \begin{array}{lr}
       1 & \textup{if } s_i=+,\\
       2 &  \textup{if } s_i=-.
            \end{array}
   \right.
\end{displaymath} 
This section studies semistandard fillings of $(3,\ldots , 3) \vdash 3n$ of content $\lambda_s$ which we will refer to as fillings of content $s$. For example, the first tableau in Figure \ref{semitostd} has content $s=--+++++$. Note that the number of fillings of content $s$ depends only on the number of pluses and minuses and not on the order in which these symbols appear.

The bijection described in this section works for all sign strings. For ease of notation, we provide explicit instructions for the case that $s = - \cdots - + \cdots +$. The general case is similar. An example is given at the end of  Section \ref{examples}.

Let $s = - \cdots - + \cdots +$ be a sign string consisting of $k$ minuses followed by $3n-k$ pluses. Let $T_s$ be a filling of $(3,\ldots, 3) \vdash 3n$ of content $s$. From $T_s$, construct a standard tableau $\widetilde{T_s}$ by replacing each repeated pair $i, i$ with the numbers $2i-1, 2i$ such that $2i-1$ is to the left of $2i$; for $i>k$, replace $i$ with $i+k$. Figure \ref{semitostd} has an example. For any tableau $T$, write $T'$ for its conjugate. 

\begin{figure}[h]
	$T_s =\raisebox{-15pt}{\young(112,235,467)} \hspace{.25in} \longrightarrow \hspace{.25in} \widetilde{T_s} = \raisebox{-15pt}{\young(124,357,689)} 
	\hspace{.25in} \longrightarrow \hspace{.25in}  \widetilde{T_s}' = \raisebox{-15pt}{\young(136,258,479)}$
	\caption{Obtaining a standard from a semistandard and taking its conjugate.}\label{semitostd}
\end{figure}

 \begin{lemma}
Let $T_s$ be a tableau of content $s$. Then $(2i-1,2i)\in \tau(\widetilde{T_s}')$ for all $1\leq i\leq k$.
 \end{lemma}
 \begin{proof}
 Since $T_s$ is semistandard, the repeated pairs $i,i$ can never be in the same column of $T_s$. To construct $\widetilde{T_s}$ the leftmost instance of $i$ is replaced with $2i-1$, and the rightmost instance is replaced with $2i$. This means that $2i-1$ will always lie in a row above $2i$ in the conjugate tableau $\widetilde{T_s}'$.
 \end{proof}
 
 Let $w_{\widetilde{T_s}'}$ be the non-elliptic web constructed from $\widetilde{T_s}'$ using the $m$-diagram algorithm in Section \ref{JTalg}. Let $w_{T_s}$ be the web formed by contracting the first $2k$ boundary edges of $w_{\widetilde{T_s}'}$ as shown in Figure \ref{contract}.
 
 \begin{figure}[h]
 \begin{tikzpicture}[baseline=0cm, scale=.8]

\draw[style=thick](5,0)--(8,0);
\draw[radius=.08, fill=black](6,0)circle;
\draw[radius=.08, fill=black](7,0)circle;
\draw[radius=.08, fill=black](6.5,.5)circle;
\draw[radius=.08, fill=black](6.5,1)circle;
\node at (6,-.3) {$2i-1$};
\node at (7,-.3) {$2i$};
\draw[style=thick,->](6,0)--(6.25,.25);
\draw[style=thick](6.25,.25)--(6.5, .5);
\draw[style=thick,->](7,0)--(6.75,.25);
\draw[style=thick](6.75,.25)--(6.5, .5);
\draw[style=thick,-<](6.5,.5)--(6.5,.75);
\draw[style=thick](6.5,.75)--(6.5,1);
\draw[style=thick,->](6.5,1)--(6.25,1.1);
\draw[style=thick](6.25,1.1)--(6,1.2);
\draw[style=thick,->](6.5,1)--(6.75,1.1);
\draw[style=thick](6.75,1.1)--(7,1.2);

\draw[style=ultra thick, ->](8.5, .5)--(9.5, .5);

\draw[style=thick](10,0)--(13,0);
\draw[radius=.08, fill=black](11.5,0)circle;
\draw[radius=.08, fill=black](11.5,1)circle;
\draw[style=thick, -<](11.5,0)--(11.5,.5);
\draw[style=thick](11.5,.5)--(11.5,1);
\draw[style=thick,->](11.5,1)--(11.75,1.1);
\draw[style=thick](11.75,1.1)--(12,1.2);
\draw[style=thick,->](11.5,1)--(11.25,1.1);
\draw[style=thick](11.25,1.1)--(11,1.2);
\node at (11.5, -.3) {$i$};

\end{tikzpicture}
 \caption{Contracting boundary edges to produce a sink vertex.} \label{contract}
 \end{figure}
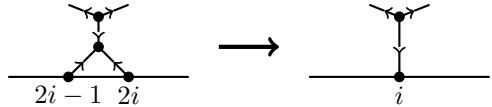
 
 \begin{lemma}
 The sign string associated to the boundary of $w_{T_s}$ is $s$.
 \end{lemma}
 \begin{proof}
 Since the construction of $w_{T_s}$ leaves the last $3n-k$ boundary vertices undisturbed, it is clear that they will be sources since they were sources in $w_{\widetilde{T_s}'}$. Since $(2i-1, 2i)\in \tau(\widetilde{T_s}')$ for all $1\leq i\leq k$, vertices $2i-1$ and $2i$ are connected to the same internal vertex. When we contract the boundary edges incident on vertices $2i-1$ and $2i$, the internal vertex they share becomes a new boundary vertex. This vertex is a sink. Thus, each pair of vertices $2i-1, 2i$ is replaced with a single sink vertex, and the boundary of $w_{T_s}$ consists of $k$ sinks followed by $3n-k$ sources as desired.
 \end{proof}
 
 \begin{lemma}
 The web $w_{T_s}$ is non-elliptic.
 \end{lemma}
 \begin{proof}
 Any closed face of $w_{T_s}$ is also a closed face of $w_{\widetilde{T_s}'}$. Since $w_{\widetilde{T_s}'}$ is non-elliptic, it follows that $w_{T_s}$ is also non-elliptic.
 \end{proof}
 
 \begin{lemma}
 Given two different fillings $T_{s,1}$ and $T_{s,2}$ of content $s$, the webs $w_{T_{s,1}}$ and $w_{T_{s,2}}$ are distinct. 
 \end{lemma}
 \begin{proof}
 If $T_{s,1}$ and $T_{s,2}$ differ on some repeated number $i$ where $1\leq i \leq k$ then the standard tableaux $\widetilde{T_{s,1}}$ and $\widetilde{T_{s,2}}$ will have at least one of the pair $2i-1, 2i$ in different positions.  If $T_{s,1}$ and $T_{s,2}$ differ on some unrepeated number $i$ where $k<i\leq 3n-k$ then the number $i+k$ will be in different positions in $\widetilde{T_{s,1}}$ and $\widetilde{T_{s,2}}$. In either case $\widetilde{T_{s,1}}$ and $\widetilde{T_{s,2}}$ are distinct.
 
 Since $\widetilde{T_{s,1}}$ and $\widetilde{T_{s,2}}$ must be distinct, it follows that $w_{\widetilde{T_{s,1}}'}$ and $w_{\widetilde{T_{s,2}}'}$ are distinct webs with $3n$ sources. The portions of $w_{\widetilde{T_{s,1}}'}$ and $w_{\widetilde{T_{s,2}}'}$ that are contracted to form $w_{T_{s,1}}$ and $w_{T_{s,2}}$ are identical according to Lemma \ref{tausetlemma}. Therefore $w_{\widetilde{T_{s,1}}'}$ and $w_{\widetilde{T_{s,2}}'}$ must differ away from the first $2k$ boundary edges which means that  $w_{T_{s,1}}$ and $w_{T_{s,2}}$ must be distinct as well.
 \end{proof}
  
 \begin{theorem}
 The map sending fillings $T_s$ of content $s$ to webs $w_{T_s}$ with boundary $s$ is a bijection.
 \end{theorem}
 \begin{proof}
 The previous lemmas show that this map sends distinct semistandard tableaux of content $s$ to distinct webs with boundary $s$. Since these two sets are in bijection, it follows that this map gives a bijective correspondence.
 \end{proof}
 
Given a web $w$ with boundary $s$, construct a tableau ${T_{s_w}}$ of content $s$ as follows. 
 \begin{itemize}
 \item{If the depth of $F_{i,L}$ is less than the depth of $F_{i,R}$ put the pair $i,i$ in the first and second column if vertex $i$ of $w$ is a sink and $i$ in the first column if it is a source.} 
 \item{If $F_{i,L}$ has the same depth as $F_{i,R}$ then put the pair $i,i$ in the first and third column if $i$ is a sink and $i$ in the second column if it is a source.}
 \item{If the depth of $F_{i,L}$ is greater than $F_{i,R}$ then put the pair $i,i$ in the second and third column if $i$ is a sink and $i$ in the third column if it is a source.}
 \end{itemize}
By construction this process is inverse to the algorithm given above for building a web from a semistandard tableau.

\section{Examples} \label{examples}
Consider the sign string $s = ++-++-+$. The composition $\lambda_s$ in this case is $\lambda_s = (1, 1, 2, 1, 1, 2, 1)$. Then the following tableau $T_s$ which is a filling using the numbers $\{1, 2, 3^2, 4, 5, 6^2, 7\}$ is said to have content $s$. We construct the two standard tableaux $\widetilde{T_s}$ and $\widetilde{T_s}'$ using the natural generalization of the algorithm from the previous section. In particular, we replace the ordered set $\{1, 2, 3, 3, 4, 5, 6, 6, 7\}$ with the ordered set $\{1,2, 3, 4, 5, 6, 7, 8, 9\}$ always placing the smaller number farthest left when replacing a repeated pair.
$$T_s = \raisebox{-15pt}{\young(134,256,367)} \hspace{.25in} \longrightarrow \hspace{.25in} \widetilde{T_s} = \raisebox{-15pt}{\young(145,268,379)} \hspace{.25in} \longrightarrow \hspace{.25in} \widetilde{T_s}' = \raisebox{-15pt}{\young(123,467,589)}$$
From the tableau $\widetilde{T_s}'$ we get the $m$-diagram and the web shown in Figure \ref{niceweb}. After contracting the edges incident on vertices 3, 4, 7, and 8 we get the web with boundary $s$ shown in Figure \ref{nicefinalweb}.
\begin{figure}[h]
\raisebox{-35pt}{$m_{\widetilde{T_s}'} =$} $\raisebox{-60pt}{ \begin{tikzpicture}[baseline=0cm, scale=.5]
\draw[style=thick] (.5,0)--(9.5,0);
\draw[radius=.08, fill=black](1,0)circle;
\draw[radius=.08, fill=black](2,0)circle;
\draw[radius=.08, fill=black](3,0)circle;
\draw[radius=.08, fill=black](4,0)circle;
\draw[radius=.08, fill=black](5,0)circle;
\draw[radius=.08, fill=black](6,0)circle;
\draw[radius=.08, fill=black](7,0)circle;
\draw[radius=.08, fill=black](8,0)circle;
\draw[radius=.08, fill=black](9,0)circle;
\draw (4,0) arc (0:180: .5cm);
\draw (6,0) arc (0:180: 2cm);
\draw (7,0) arc (0:180: 3cm);
\draw (5,0) arc (0:180: .5cm);
\draw (8,0) arc (0:180: .5cm);
\draw (9,0) arc (0:180: 1.5cm);
\node at (1,-.5) {1};
\node at (2,-.5) {2};
\node at (3,-.5) {3};
\node at (4,-.5) {4};
\node at (5,-.5) {5};
\node at (6,-.5) {6};
\node at (7,-.5) {7};
\node at (8,-.5) {8};
\node at (9,-.5) {9};
\end{tikzpicture}}$ 
\hspace{1in}
\raisebox{-35pt}{$w_{\widetilde{T_s}'} =$}  \raisebox{-60pt}{\begin{tikzpicture}[baseline=0cm, scale=.5]
\draw[style=thick] (.5,0)--(9.5,0);
\draw[radius=.08, fill=black](1,0)circle;
\draw[radius=.08, fill=black](2,0)circle;
\draw[radius=.08, fill=black](3,0)circle;
\draw[radius=.08, fill=black](4,0)circle;
\draw[radius=.08, fill=black](5,0)circle;
\draw[radius=.08, fill=black](6,0)circle;
\draw[radius=.08, fill=black](7,0)circle;
\draw[radius=.08, fill=black](8,0)circle;
\draw[radius=.08, fill=black](9,0)circle;

\draw[ style=thick,->](4, 0) -- (4,.5);
\draw[ style=thick](4,.5)--(4,1);
\draw[radius=.08, fill=black](4,1)circle;
\draw[style=thick,->](3,0)--(3.5,.5);
\draw[style=thick](3.5,.5)--(4,1);
\draw[style=thick,->](5,0)--(4.5,.5);
\draw[style=thick](4.5,.5)--(4,1);
\draw[ style=thick,->](7, 0) -- (7,.5);
\draw[ style=thick](7,.5)--(7,1);
\draw[radius=.08, fill=black](7,1)circle;
\draw[ style=thick,->](6, 0) -- (6,1.5);
\draw[ style=thick](6,1.5)--(6,3);
\draw[radius=.08, fill=black](6,3)circle;
\draw[style=thick,->](8,0)--(7.5,.5);
\draw[style=thick](7.5,.5)--(7,1);
\draw[style=thick](7,1) -- (6.75,2.5);
\draw[style=thick, <-](6.75,2.5)--(6.5, 4);
\draw[radius=.08, fill=black](6.5,4)circle;
\draw[style=thick](6,3) -- (6.25,3.5);
\draw[style=thick, <-](6.25,3.5)--(6.5, 4);
\draw[style=thick, ->](2,0)--(4, 1.5);
\draw[style=thick](4, 1.5) -- (6, 3);
\draw[style=thick,->](6.5,4) -- (6.5,4.5);
\draw[style=thick](6.5,4.5)--(6.5,5);
\draw[radius=.08, fill=black](6.5,5)circle;
\draw[style=thick, ->](1,0)--(3.75,2.5);
\draw[style=thick](3.75,2.5)--(6.5,5);
\draw[style=thick,->](9,0)--(7.75,2.5);
\draw[style=thick](7.75,2.5)--(6.5,5);
\node at (1,-.5) {1};
\node at (2,-.5) {2};
\node at (3,-.5) {3};
\node at (4,-.5) {4};
\node at (5,-.5) {5};
\node at (6,-.5) {6};
\node at (7,-.5) {7};
\node at (8,-.5) {8};
\node at (9,-.5) {9};
\end{tikzpicture}}
\caption{The $m$-diagram and web corresponding to $\widetilde{T_s}'$.}\label{niceweb}
\end{figure}
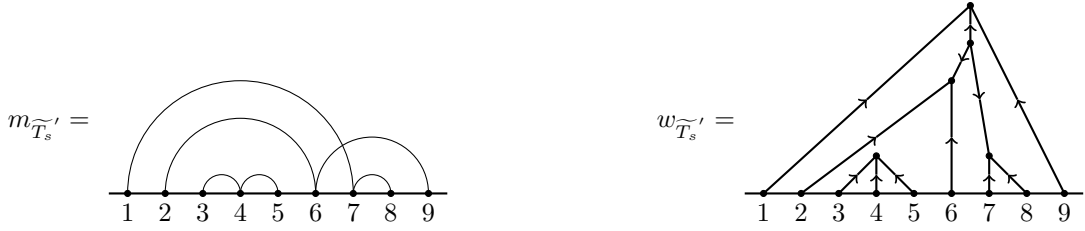

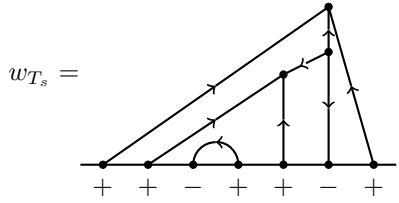
\begin{figure}[h]
\raisebox{-8pt}{$w_{T_s} =  $}\raisebox{-40pt}{\begin{tikzpicture}[baseline=0cm, scale=.6]
\draw[style=thick] (.5,0)--(7.5,0);
\draw[radius=.08, fill=black](1,0)circle;
\draw[radius=.08, fill=black](2,0)circle;
\draw[radius=.08, fill=black](3,0)circle;
\draw[radius=.08, fill=black](4,0)circle;
\draw[radius=.08, fill=black](5,0)circle;
\draw[radius=.08, fill=black](6,0)circle;
\draw[radius=.08, fill=black](7,0)circle;

\draw[style=thick, ->](4,0) arc (0:90:.5cm);
\draw[style=thick](3.5,.5) arc (90:180:.5cm);
\draw[style=thick, ->](5,0)--(5,1);
\draw[style=thick](5,1)--(5,2);
\draw[radius=.08, fill=black](5,2)circle;
\draw[style=thick,->](2,0) -- (3.5,1);
\draw[style=thick](3.5,1)--(5,2);
\draw[style=thick,-<](5,2)--(5.5,2.25);
\draw[style=thick](5.5,2.25)--(6,2.5);
\draw[radius=.08, fill=black](6,2.5)circle;
\draw[style=thick,->](6,2.5) -- (6,1.25);
\draw[style=thick](6,1.25)--(6,0);
\draw[style=thick,->](6,2.5)--(6,3);
\draw[style=thick](6,3)--(6,3.5);
\draw[radius=.08, fill=black](6,3.5)circle;
\draw[style=thick,->](7,0)--(6.5,1.75);
\draw[style=thick](6.5,1.75)--(6,3.5);
\draw[style=thick,->](1,0)--(3.5,1.75);
\draw[style=thick](3.5,1.75)--(6,3.5);
\node at (1,-.5) {+};
\node at (2,-.5) {+};
\node at (3,-.5) {$-$};
\node at (4,-.5) {+};
\node at (5,-.5) {+};
\node at (6,-.5) {$-$};
\node at (7,-.5) {+};
\end{tikzpicture}}
\caption{The web of content $s$ corresponding to $T_s$.}\label{nicefinalweb}
\end{figure}

We conclude this section with some additional examples. Figures \ref{s1},\ref{s2}, and \ref{s3} construct all webs corresponding to the sign strings $---, --++,$ and $-++++$ respectively. 

\begin{figure}[h]
\begin{tabular}{|c|c|c|c|c|c|}
\hline &&&&& \\
$T_{s}$ & $\widetilde{T_{s}}$ & $\widetilde{T_{s}}'$ & $m$-diagram & $w_{\widetilde{T_{s}}'}$ & $w_{T_{s}}$ \\
&&&&&\\ \hline
 &&&&&\\
\young(112,233)&\young(124,356)& \raisebox{-8pt}{\young(13,25,46) }& \raisebox{10pt}{\begin{tikzpicture}[baseline=0cm, scale=.6]
\draw[style=thick] (3,-.7)--(9,-.7);
\draw[radius=.08, fill=black](3.5,-.7)circle;
\draw[radius=.08, fill=black](4.5,-.7)circle;
\draw[radius=.08, fill=black](5.5,-.7)circle;
\draw[radius=.08, fill=black](6.5,-.7)circle;
\draw[radius=.08, fill=black](7.5,-.7)circle;
\draw[radius=.08, fill=black](8.5,-.7)circle;
\draw (4.5,-.7) arc (0:180: .5cm);
\draw (7.5,-.7) arc (0:180: 1cm);
\draw (6.5,-.7) arc (0:180: 1cm);
\draw (8.5,-.7) arc (0:180: .5cm);
\node at (3.5,-1.1) {1};
\node at (4.5,-1.1) {2};
\node at (5.5,-1.1) {3};
\node at (6.5,-1.1) {4};
\node at (7.5,-1.1) {5};
\node at (8.5,-1.1) {6};
\end{tikzpicture}} &
\raisebox{10pt}{\begin{tikzpicture}[baseline=0cm, scale=.6]
\draw[style=thick] (2.5,-.7)--(8.5,-.7);
\draw[radius=.08, fill=black](3,-.7)circle;
\draw[radius=.08, fill=black](4,-.7)circle;
\draw[radius=.08, fill=black](5,-.7)circle;
\draw[radius=.08, fill=black](6,-.7)circle;
\draw[radius=.08, fill=black](7,-.7)circle;
\draw[radius=.08, fill=black](8,-.7)circle;
\draw[yshift=-.7cm, style=thick,->](4, 0) -- (4,.5);
\draw[yshift=-.7cm, style=thick](4,.5)--(4,1);
\draw[yshift=-.7cm, radius=.08, fill=black](4,1)circle;
\draw[yshift=-.7cm, style=thick,->](3,0)--(3.5,.5);
\draw[yshift=-.7cm, style=thick](3.5,.5)--(4,1);
\draw[yshift=-.7cm, style=thick,->](7, 0) -- (7,.5);
\draw[yshift=-.7cm, style=thick](7,.5)--(7,1);
\draw[yshift=-.7cm, radius=.08, fill=black](7,1)circle;
\draw[yshift=-.7cm, style=thick,->](8,0)--(7.5,.5);
\draw[yshift=-.7cm, style=thick](7.5,.5)--(7,1);
\draw[yshift=-.7cm, radius=.08, fill=black](5.5,1)circle;
\draw[yshift=-.7cm, style=thick,-<](5.5,1)--(5.5,1.5);
\draw[yshift=-.7cm, style=thick](5.5,1.5)--(5.5,2);
\draw[yshift=-.7cm, radius=.08, fill=black](5.5,2)circle;
\draw[yshift=-.7cm, style=thick,->](5,0)--(5.25,.5);
\draw[yshift=-.7cm, style=thick](5.25,.5)--(5.5,1);
\draw[yshift=-.7cm, style=thick,->](6,0)--(5.75,.5);
\draw[yshift=-.7cm, style=thick](5.75,.5)--(5.5,1);
\draw[yshift=-.7cm, style=thick,-<](4,1)--(4.75,1.5);
\draw[yshift=-.7cm, style=thick](4.75,1.5)--(5.5,2);
\draw[yshift=-.7cm, style=thick,-<](7,1)--(6.25,1.5);
\draw[yshift=-.7cm, style=thick](6.25,1.5)--(5.5,2);
\end{tikzpicture}}   & \raisebox{-2pt}{\begin{tikzpicture}[baseline=0cm, scale=.6]
\draw[style=thick] (5.5,0)--(8.5,0);
\draw[radius=.08, fill=black](6,0)circle;
\draw[radius=.08, fill=black](7,0)circle;
\draw[radius=.08, fill=black](8,0)circle;

\draw[xshift=3cm, style=thick,-<](4, 0) -- (4,.5);
\draw[xshift=3cm, style=thick](4,.5)--(4,1);
\draw[xshift=3cm, radius=.08, fill=black](4,1)circle;
\draw[xshift=3cm,  style=thick,-<](3,0)--(3.5,.5);
\draw[xshift=3cm, style=thick](3.5,.5)--(4,1);
\draw[xshift=3cm, style=thick,-<](5,0)--(4.5,.5);
\draw[xshift=3cm, style=thick](4.5,.5)--(4,1);
\node at (6,-.5) {$-$};
\node at (7,-.5) {$-$};
\node at (8,-.5) {$-$};
\end{tikzpicture}} \\ \hline
 \end{tabular}
 \caption{Constructing the bijection for sign string $s=---$.}\label{s1}
 \end{figure}
 
 \begin{figure}[h]
\begin{tabular}{|c|c|c|c|c|c|}
\hline &&&&& \\
$T_{s}$ & $\widetilde{T_{s}}$ & $\widetilde{T_{s}}'$ & $m$-diagram & $w_{\widetilde{T_{s}}'}$ & $w_{T_{s}}$ \\
&&&&&\\ \hline
&&&&&\\
 \young(112,234) & \young(124,356) & \raisebox{-8pt}{\young(13,25,46) }& \raisebox{10pt}{\begin{tikzpicture}[baseline=0cm, scale=.6]
\draw[style=thick] (3,-.7)--(9,-.7);
\draw[radius=.08, fill=black](3.5,-.7)circle;
\draw[radius=.08, fill=black](4.5,-.7)circle;
\draw[radius=.08, fill=black](5.5,-.7)circle;
\draw[radius=.08, fill=black](6.5,-.7)circle;
\draw[radius=.08, fill=black](7.5,-.7)circle;
\draw[radius=.08, fill=black](8.5,-.7)circle;
\draw (4.5,-.7) arc (0:180: .5cm);
\draw (7.5,-.7) arc (0:180: 1cm);
\draw (6.5,-.7) arc (0:180: 1cm);
\draw (8.5,-.7) arc (0:180: .5cm);
\node at (3.5,-1.1) {1};
\node at (4.5,-1.1) {2};
\node at (5.5,-1.1) {3};
\node at (6.5,-1.1) {4};
\node at (7.5,-1.1) {5};
\node at (8.5,-1.1) {6};
\end{tikzpicture}} &
\raisebox{10pt}{\begin{tikzpicture}[baseline=0cm, scale=.6]
\draw[style=thick] (2.5,-.7)--(8.5,-.7);
\draw[radius=.08, fill=black](3,-.7)circle;
\draw[radius=.08, fill=black](4,-.7)circle;
\draw[radius=.08, fill=black](5,-.7)circle;
\draw[radius=.08, fill=black](6,-.7)circle;
\draw[radius=.08, fill=black](7,-.7)circle;
\draw[radius=.08, fill=black](8,-.7)circle;
\draw[yshift=-.7cm, style=thick,->](4, 0) -- (4,.5);
\draw[yshift=-.7cm, style=thick](4,.5)--(4,1);
\draw[yshift=-.7cm, radius=.08, fill=black](4,1)circle;
\draw[yshift=-.7cm, style=thick,->](3,0)--(3.5,.5);
\draw[yshift=-.7cm, style=thick](3.5,.5)--(4,1);
\draw[yshift=-.7cm, style=thick,->](7, 0) -- (7,.5);
\draw[yshift=-.7cm, style=thick](7,.5)--(7,1);
\draw[yshift=-.7cm, radius=.08, fill=black](7,1)circle;
\draw[yshift=-.7cm, style=thick,->](8,0)--(7.5,.5);
\draw[yshift=-.7cm, style=thick](7.5,.5)--(7,1);
\draw[yshift=-.7cm, radius=.08, fill=black](5.5,1)circle;
\draw[yshift=-.7cm, style=thick,-<](5.5,1)--(5.5,1.5);
\draw[yshift=-.7cm, style=thick](5.5,1.5)--(5.5,2);
\draw[yshift=-.7cm, radius=.08, fill=black](5.5,2)circle;
\draw[yshift=-.7cm, style=thick,->](5,0)--(5.25,.5);
\draw[yshift=-.7cm, style=thick](5.25,.5)--(5.5,1);
\draw[yshift=-.7cm, style=thick,->](6,0)--(5.75,.5);
\draw[yshift=-.7cm, style=thick](5.75,.5)--(5.5,1);
\draw[yshift=-.7cm, style=thick,-<](4,1)--(4.75,1.5);
\draw[yshift=-.7cm, style=thick](4.75,1.5)--(5.5,2);
\draw[yshift=-.7cm, style=thick,-<](7,1)--(6.25,1.5);
\draw[yshift=-.7cm, style=thick](6.25,1.5)--(5.5,2);
\end{tikzpicture}}   &\raisebox{-2pt}{ \begin{tikzpicture}[baseline=0cm, scale=.6]
\draw[style=thick] (4.5,0)--(8.5,0);

\draw[radius=.08, fill=black](5,0)circle;
\draw[radius=.08, fill=black](6,0)circle;
\draw[radius=.08, fill=black](7,0)circle;
\draw[radius=.08, fill=black](8,0)circle;

\node at (5,-.5) {$-$};
\node at (6,-.5) {$-$};
\node at (7,-.5) {+};
\node at (8,-.5){+};

\draw[style=thick,->](7, 0) -- (7,.5);
\draw[style=thick](7,.5)--(7,1);
\draw[style=thick,-<](6, 0) -- (6,.5);
\draw[style=thick](6,.5)--(6,1);
\draw[radius=.08, fill=black](7,1)circle;
\draw[radius=.08, fill=black](6,1)circle;
\draw[style=thick,->](8,0)--(7.5,.5);
\draw[style=thick](7.5,.5)--(7,1);
\draw[style=thick, ->] (6,1) -- (6.5,1);
\draw[style=thick](6.5,1)--(7,1);

\draw[style=thick,-<](5,0)--(5.5,.5);
\draw[style=thick](5.5,.5)--(6,1);

\end{tikzpicture}}\\ 
 &&&&& \\\hline
 &&&&& \\
 \young(113,224) & \young(125,346) & \raisebox{-8pt}{\young(13,24,56) }& \raisebox{10pt}{\begin{tikzpicture}[baseline=0cm, scale=.6]
\draw[style=thick] (3,-.7)--(9,-.7);
\draw[radius=.08, fill=black](3.5,-.7)circle;
\draw[radius=.08, fill=black](4.5,-.7)circle;
\draw[radius=.08, fill=black](5.5,-.7)circle;
\draw[radius=.08, fill=black](6.5,-.7)circle;
\draw[radius=.08, fill=black](7.5,-.7)circle;
\draw[radius=.08, fill=black](8.5,-.7)circle;
\draw (4.5,-.7) arc (0:180: .5cm);
\draw (8.5,-.7) arc (0:180: 2cm);
\draw (6.5,-.7) arc (0:180: .5cm);
\draw (7.5,-.7) arc (0:180: .5cm);
\node at (3.5,-1.1) {1};
\node at (4.5,-1.1) {2};
\node at (5.5,-1.1) {3};
\node at (6.5,-1.1) {4};
\node at (7.5,-1.1) {5};
\node at (8.5,-1.1) {6};
\end{tikzpicture}} &
\raisebox{10pt}{\begin{tikzpicture}[baseline=0cm, scale=.6]
\draw[yshift=-.7cm, style=thick] (2.5,0)--(8.5,0);
\draw[yshift=-.7cm, radius=.08, fill=black](3,0)circle;
\draw[yshift=-.7cm, radius=.08, fill=black](4,0)circle;
\draw[yshift=-.7cm, radius=.08, fill=black](5,0)circle;
\draw[yshift=-.7cm, radius=.08, fill=black](6,0)circle;
\draw[yshift=-.7cm, radius=.08, fill=black](7,0)circle;
\draw[yshift=-.7cm, radius=.08, fill=black](8,0)circle;

\draw[yshift=-.7cm, style=thick,->](4, 0) -- (4,1);
\draw[yshift=-.7cm, style=thick](4,.75)--(4,2);
\draw[yshift=-.7cm, radius=.08, fill=black](4,2)circle;
\draw[yshift=-.7cm, style=thick,->](3,0)--(3.5,1);
\draw[yshift=-.7cm, style=thick](3.5,1)--(4,2);
\draw[yshift=-.7cm, style=thick,->](8,0)--(6,1);
\draw[yshift=-.7cm,style=thick](6,1)--(4,2);

\draw[xshift=2cm, yshift=-.7cm, style=thick,->](4, 0) -- (4,.25);
\draw[xshift=2cm, yshift=-.7cm, style=thick](4,.2)--(4,.5);
\draw[xshift=2cm, yshift=-.7cm, radius=.08, fill=black](4,.5)circle;
\draw[xshift=2cm, yshift=-.7cm, style=thick,->](3,0)--(3.5,.25);
\draw[xshift=2cm, yshift=-.7cm, style=thick](3.5,.25)--(4,.5);
\draw[xshift=2cm, yshift=-.7cm, style=thick,->](5,0)--(4.5,.25);
\draw[xshift=2cm, yshift=-.7cm, style=thick](4.5,.25)--(4,.5);
\end{tikzpicture}} &  \raisebox{-2pt}{\begin{tikzpicture}[baseline=0cm, scale=.6]
\draw[style=thick] (4.5,0)--(8.5,0);
\draw[radius=.08, fill=black](5,0)circle;
\draw[radius=.08, fill=black](6,0)circle;
\draw[radius=.08, fill=black](7,0)circle;
\draw[radius=.08, fill=black](8,0)circle;

\node at (5,-.5) {$-$};
\node at (6,-.5) {$-$};
\node at (7,-.5) {+};
\node at (8,-.5){+};

\draw[style=thick, ->] (8,0) arc (0:90: 1.5cm);
\draw[style=thick](6.5,1.5) arc (90:180:1.5cm);
\draw[style=thick, ->] (7,0) arc (0:90: .5cm);
\draw[style=thick](6.5,.5) arc (90:180:.5cm);
\end{tikzpicture}}  \\  
&&&&& \\ \hline
\end{tabular}
\caption{Constructing the bijection for sign string $s=--++$.}\label{s2}
\end{figure}

\begin{figure}[h]
\begin{tabular}{|c|c|c|c|c|c|}
\hline &&&&& \\
$T_{s}$ & $\widetilde{T_{s}}$ & $\widetilde{T_{s}}'$ & $m$-diagram & $w_{\widetilde{T_{s}}'}$ & $w_{T_{s}}$ \\
&&&&&\\  \hline
&&&&&\\
 \young(112,345) & \young(123,456) & \raisebox{-8pt}{\young(14,25,36) }& \raisebox{10pt}{\begin{tikzpicture}[baseline=0cm, scale=.6]
\draw[style=thick] (3,-.7)--(9,-.7);
\draw[radius=.08, fill=black](3.5,-.7)circle;
\draw[radius=.08, fill=black](4.5,-.7)circle;
\draw[radius=.08, fill=black](5.5,-.7)circle;
\draw[radius=.08, fill=black](6.5,-.7)circle;
\draw[radius=.08, fill=black](7.5,-.7)circle;
\draw[radius=.08, fill=black](8.5,-.7)circle;
\draw (4.5,-.7) arc (0:180: .5cm);
\draw (6.5,-.7) arc (180:0: .5cm);
\draw (5.5,-.7) arc (0:180: .5cm);
\draw (8.5,-.7) arc (0:180: .5cm);
\node at (3.5,-1.1) {1};
\node at (4.5,-1.1) {2};
\node at (5.5,-1.1) {3};
\node at (6.5,-1.1) {4};
\node at (7.5,-1.1) {5};
\node at (8.5,-1.1) {6};
\end{tikzpicture}} & \raisebox{10pt}{ \begin{tikzpicture}[baseline=0cm, scale=.6]
\draw[yshift=-.7cm, style=thick] (2.5,0)--(8.5,0);
\draw[yshift=-.7cm, radius=.08, fill=black](3,0)circle;
\draw[yshift=-.7cm, radius=.08, fill=black](4,0)circle;
\draw[yshift=-.7cm, radius=.08, fill=black](5,0)circle;
\draw[yshift=-.7cm, radius=.08, fill=black](6,0)circle;
\draw[yshift=-.7cm, radius=.08, fill=black](7,0)circle;
\draw[yshift=-.7cm, radius=.08, fill=black](8,0)circle;

\draw[yshift=-.7cm, style=thick,->](4, 0) -- (4,.5);
\draw[yshift=-.7cm, style=thick](4,.5)--(4,1);
\draw[yshift=-.7cm, radius=.08, fill=black](4,1)circle;
\draw[yshift=-.7cm, style=thick,->](3,0)--(3.5,.5);
\draw[yshift=-.7cm, style=thick](3.5,.5)--(4,1);
\draw[yshift=-.7cm, style=thick,->](5,0)--(4.5,.5);
\draw[yshift=-.7cm, style=thick](4.5,.5)--(4,1);

\draw[xshift=3cm, yshift=-.7cm, style=thick,->](4, 0) -- (4,.5);
\draw[xshift=3cm, yshift=-.7cm, style=thick](4,.5)--(4,1);
\draw[xshift=3cm, yshift=-.7cm, radius=.08, fill=black](4,1)circle;
\draw[xshift=3cm, yshift=-.7cm, style=thick,->](3,0)--(3.5,.5);
\draw[xshift=3cm, yshift=-.7cm, style=thick](3.5,.5)--(4,1);
\draw[xshift=3cm, yshift=-.7cm, style=thick,->](5,0)--(4.5,.5);
\draw[xshift=3cm, yshift=-.7cm, style=thick](4.5,.5)--(4,1);
\end{tikzpicture}} & \raisebox{-2pt}{\begin{tikzpicture}[baseline=0cm, scale=.6]
\draw[style=thick] (3.5,0)--(8.5,0);
\draw[radius=.08, fill=black](4,0)circle;
\draw[radius=.08, fill=black](5,0)circle;
\draw[radius=.08, fill=black](6,0)circle;
\draw[radius=.08, fill=black](7,0)circle;
\draw[radius=.08, fill=black](8,0)circle;
\draw[style=thick, ->] (5,0) arc (0:90: .5cm);
\draw[style=thick](4.5,.5) arc (90:180:.5cm);
\draw[xshift=3cm, style=thick,->](4, 0) -- (4,.5);
\draw[xshift=3cm, style=thick](4,.5)--(4,1);
\draw[xshift=3cm, radius=.08, fill=black](4,1)circle;
\draw[xshift=3cm,  style=thick,->](3,0)--(3.5,.5);
\draw[xshift=3cm, style=thick](3.5,.5)--(4,1);
\draw[xshift=3cm, style=thick,->](5,0)--(4.5,.5);
\draw[xshift=3cm, style=thick](4.5,.5)--(4,1);
\node at (4,-.5) {$-$};
\node at (5,-.5) {+};
\node at (6,-.5) {+};
\node at (7,-.5) {+};
\node at (8,-.5){+};
\end{tikzpicture}}\\
&&&&& \\
\hline &&&&&\\
\young(114,235)& \young(125,346) &\raisebox{-8pt}{ \young(13,24,56)}& \raisebox{10pt}{\begin{tikzpicture}[baseline=0cm, scale=.6]
\draw[style=thick] (3,-.7)--(9,-.7);
\draw[radius=.08, fill=black](3.5,-.7)circle;
\draw[radius=.08, fill=black](4.5,-.7)circle;
\draw[radius=.08, fill=black](5.5,-.7)circle;
\draw[radius=.08, fill=black](6.5,-.7)circle;
\draw[radius=.08, fill=black](7.5,-.7)circle;
\draw[radius=.08, fill=black](8.5,-.7)circle;
\draw (4.5,-.7) arc (0:180: .5cm);
\draw (8.5,-.7) arc (0:180: 2cm);
\draw (6.5,-.7) arc (0:180: .5cm);
\draw (7.5,-.7) arc (0:180: .5cm);
\node at (3.5,-1.1) {1};
\node at (4.5,-1.1) {2};
\node at (5.5,-1.1) {3};
\node at (6.5,-1.1) {4};
\node at (7.5,-1.1) {5};
\node at (8.5,-1.1) {6};
\end{tikzpicture}} &
\raisebox{10pt}{\begin{tikzpicture}[baseline=0cm, scale=.6]
\draw[yshift=-.7cm, style=thick] (2.5,0)--(8.5,0);
\draw[yshift=-.7cm, radius=.08, fill=black](3,0)circle;
\draw[yshift=-.7cm, radius=.08, fill=black](4,0)circle;
\draw[yshift=-.7cm, radius=.08, fill=black](5,0)circle;
\draw[yshift=-.7cm, radius=.08, fill=black](6,0)circle;
\draw[yshift=-.7cm, radius=.08, fill=black](7,0)circle;
\draw[yshift=-.7cm, radius=.08, fill=black](8,0)circle;

\draw[yshift=-.7cm, style=thick,->](4, 0) -- (4,1);
\draw[yshift=-.7cm, style=thick](4,.75)--(4,2);
\draw[yshift=-.7cm, radius=.08, fill=black](4,2)circle;
\draw[yshift=-.7cm, style=thick,->](3,0)--(3.5,1);
\draw[yshift=-.7cm, style=thick](3.5,1)--(4,2);
\draw[yshift=-.7cm, style=thick,->](8,0)--(6,1);
\draw[yshift=-.7cm,style=thick](6,1)--(4,2);

\draw[xshift=2cm, yshift=-.7cm, style=thick,->](4, 0) -- (4,.25);
\draw[xshift=2cm, yshift=-.7cm, style=thick](4,.2)--(4,.5);
\draw[xshift=2cm, yshift=-.7cm, radius=.08, fill=black](4,.5)circle;
\draw[xshift=2cm, yshift=-.7cm, style=thick,->](3,0)--(3.5,.25);
\draw[xshift=2cm, yshift=-.7cm, style=thick](3.5,.25)--(4,.5);
\draw[xshift=2cm, yshift=-.7cm, style=thick,->](5,0)--(4.5,.25);
\draw[xshift=2cm, yshift=-.7cm, style=thick](4.5,.25)--(4,.5);
\end{tikzpicture}} &\raisebox{-2pt}{ \begin{tikzpicture}[baseline=0cm, scale=.6]
\draw[style=thick] (3.5,0)--(8.5,0);

\draw[radius=.08, fill=black](4,0)circle;
\draw[radius=.08, fill=black](5,0)circle;
\draw[radius=.08, fill=black](6,0)circle;
\draw[radius=.08, fill=black](7,0)circle;
\draw[radius=.08, fill=black](8,0)circle;

\draw[style=thick, ->] (8,0) arc (0:90: 2cm);
\draw[style=thick](6,2) arc (90:180:2cm);

\draw[xshift=2cm, style=thick,->](4, 0) -- (4,.25);
\draw[xshift=2cm, style=thick](4,.2)--(4,.5);
\draw[xshift=2cm, radius=.08, fill=black](4,.5)circle;
\draw[xshift=2cm, style=thick,->](3,0)--(3.5,.25);
\draw[xshift=2cm, style=thick](3.5,.25)--(4,.5);
\draw[xshift=2cm, style=thick,->](5,0)--(4.5,.25);
\draw[xshift=2cm,  style=thick](4.5,.25)--(4,.5);

\node at (4,-.5) {$-$};
\node at (5,-.5) {+};
\node at (6,-.5) {+};
\node at (7,-.5) {+};
\node at (8,-.5){+};
\end{tikzpicture}} \\ 
&&&&&\\ \hline
&&&&& \\
\young(113,245) & \young(124,356) & \raisebox{-8pt}{ \young(13,25,46)} & \raisebox{10pt}{\begin{tikzpicture}[baseline=0cm, scale=.6]
\draw[style=thick] (3,-.7)--(9,-.7);
\draw[radius=.08, fill=black](3.5,-.7)circle;
\draw[radius=.08, fill=black](4.5,-.7)circle;
\draw[radius=.08, fill=black](5.5,-.7)circle;
\draw[radius=.08, fill=black](6.5,-.7)circle;
\draw[radius=.08, fill=black](7.5,-.7)circle;
\draw[radius=.08, fill=black](8.5,-.7)circle;
\draw (4.5,-.7) arc (0:180: .5cm);
\draw (7.5,-.7) arc (0:180: 1cm);
\draw (6.5,-.7) arc (0:180: 1cm);
\draw (8.5,-.7) arc (0:180: .5cm);
\node at (3.5,-1.1) {1};
\node at (4.5,-1.1) {2};
\node at (5.5,-1.1) {3};
\node at (6.5,-1.1) {4};
\node at (7.5,-1.1) {5};
\node at (8.5,-1.1) {6};
\end{tikzpicture}} &
\raisebox{10pt}{\begin{tikzpicture}[baseline=0cm, scale=.6]
\draw[style=thick] (2.5,-.7)--(8.5,-.7);
\draw[radius=.08, fill=black](3,-.7)circle;
\draw[radius=.08, fill=black](4,-.7)circle;
\draw[radius=.08, fill=black](5,-.7)circle;
\draw[radius=.08, fill=black](6,-.7)circle;
\draw[radius=.08, fill=black](7,-.7)circle;
\draw[radius=.08, fill=black](8,-.7)circle;
\draw[yshift=-.7cm, style=thick,->](4, 0) -- (4,.5);
\draw[yshift=-.7cm, style=thick](4,.5)--(4,1);
\draw[yshift=-.7cm, radius=.08, fill=black](4,1)circle;
\draw[yshift=-.7cm, style=thick,->](3,0)--(3.5,.5);
\draw[yshift=-.7cm, style=thick](3.5,.5)--(4,1);
\draw[yshift=-.7cm, style=thick,->](7, 0) -- (7,.5);
\draw[yshift=-.7cm, style=thick](7,.5)--(7,1);
\draw[yshift=-.7cm, radius=.08, fill=black](7,1)circle;
\draw[yshift=-.7cm, style=thick,->](8,0)--(7.5,.5);
\draw[yshift=-.7cm, style=thick](7.5,.5)--(7,1);
\draw[yshift=-.7cm, radius=.08, fill=black](5.5,1)circle;
\draw[yshift=-.7cm, style=thick,-<](5.5,1)--(5.5,1.5);
\draw[yshift=-.7cm, style=thick](5.5,1.5)--(5.5,2);
\draw[yshift=-.7cm, radius=.08, fill=black](5.5,2)circle;
\draw[yshift=-.7cm, style=thick,->](5,0)--(5.25,.5);
\draw[yshift=-.7cm, style=thick](5.25,.5)--(5.5,1);
\draw[yshift=-.7cm, style=thick,->](6,0)--(5.75,.5);
\draw[yshift=-.7cm, style=thick](5.75,.5)--(5.5,1);
\draw[yshift=-.7cm, style=thick,-<](4,1)--(4.75,1.5);
\draw[yshift=-.7cm, style=thick](4.75,1.5)--(5.5,2);
\draw[yshift=-.7cm, style=thick,-<](7,1)--(6.25,1.5);
\draw[yshift=-.7cm, style=thick](6.25,1.5)--(5.5,2);
\end{tikzpicture}} & \raisebox{-2pt}{\begin{tikzpicture}[baseline=0cm, scale=.6]
\draw[style=thick] (3.5,0)--(8.5,0);

\draw[radius=.08, fill=black](4,0)circle;
\draw[radius=.08, fill=black](5,0)circle;
\draw[radius=.08, fill=black](6,0)circle;
\draw[radius=.08, fill=black](7,0)circle;
\draw[radius=.08, fill=black](8,0)circle;

\draw[style=thick, ->](5.5,2) -- (4.75,1);
\draw[style=thick](4.75,1)--(4,0);

\draw[style=thick,->](7, 0) -- (7,.5);
\draw[style=thick](7,.5)--(7,1);
\draw[radius=.08, fill=black](7,1)circle;
\draw[style=thick,->](8,0)--(7.5,.5);
\draw[style=thick](7.5,.5)--(7,1);
\draw[radius=.08, fill=black](5.5,1)circle;
\draw[style=thick,-<](5.5,1)--(5.5,1.5);
\draw[style=thick](5.5,1.5)--(5.5,2);
\draw[radius=.08, fill=black](5.5,2)circle;
\draw[style=thick,->](5,0)--(5.25,.5);
\draw[style=thick](5.25,.5)--(5.5,1);
\draw[style=thick,->](6,0)--(5.75,.5);
\draw[style=thick](5.75,.5)--(5.5,1);
\draw[style=thick,-<](7,1)--(6.25,1.5);
\draw[style=thick](6.25,1.5)--(5.5,2);

\node at (4,-.5) {$-$};
\node at (5,-.5) {+};
\node at (6,-.5) {+};
\node at (7,-.5) {+};
\node at (8,-.5){+};
\end{tikzpicture}} \\ 
&&&&& \\ \hline
\end{tabular}
\caption{Constructing the bijection for sign string $s=-++++$.}\label{s3}
\end{figure}

\section{Applications: Rotation and join of webs}
When the sign string of a web is all pluses and the corresponding tableau is standard, Petersen-Pylyavskyy-Rhoades prove that rotation of webs corresponds to jeu-de-taquin promotion~\cite{PPR}. Tymoczko uses the $m$-diagram algorithm to give a simplified proof of this fact~\cite{T}. There is also a notion of the join of two webs. Tymoczko proves that join can be understood as another move on standard tableaux called a shuffle~\cite{T}. In this section, we prove that jeu-de-taquin promotion and shuffle of semistandard tableaux correspond to rotation and join of all non-elliptic $sl_3$ webs.

\subsection{Rotation and jeu-de-taquin promotion}
Jeu-de-taquin promotion is a process on semistandard Young tableaux whereby a box (or subset of boxes) is removed, and the tableau is rearranged to form a new filling of the same shape. Figure \ref{jdt} has an example. Say that $T$ is a semistandard tableaux filled with at least one of each of the numbers $1, \ldots, \ell$. Jeu-de-taquin promotion on $T$ produces a new tableau $jdt(T)$ as follows. 
\begin{enumerate}
\item{Begin by removing the entry 1 from the top left corner of $T$.}
\item{Say that $a$ is below and $b$ is to the right of the removed box. If $a\leq b$, slide $a$ upwards into the empty position. Otherwise, slide $b$ left into the empty position.}
\item{Continue this process until the empty box has no entries to its right or below.}
\item{If 1 appears multiple times in $T$, repeat the first three steps until all 1's have been erased.}
\item{Decrement all entries by 1, and replace each empty box with $\ell$.}
\end{enumerate}

A proof of the following Lemma can be found in Sagan's book~\cite{Sagan}.
\begin{lemma}
The jeu-de-taquin process described above is well-defined on semistandard tableaux.
\end{lemma}

\begin{figure}[h]
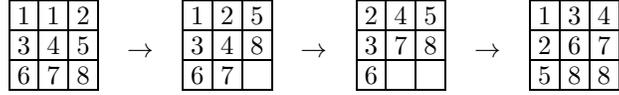

$$\young(112,345,678) \hspace{.15in} \raisebox{12pt}{$\rightarrow$} \hspace{.15in}  \young(125,348,67_{})  \hspace{.15in} \raisebox{12pt}{$\rightarrow$}  \hspace{.15in}  \young(245,378,6\hfil\hfil)  \hspace{.15in} \raisebox{12pt}{$\rightarrow$}  \hspace{.15in} \young(134,267,588)$$
\caption{Promotion on a semistandard tableau.}\label{jdt}
\end{figure}

\begin{lemma} \label{semijdt}
Let $T_s$ be a semistandard tableau of content $s = s_1\cdots s_n$. Then $\widetilde{jdt(T_s)}' = jdt(\widetilde{T_s}')$ if $s_1$ is a plus, and $\widetilde{jdt(T_s)}' = jdt\left(jdt\left(\widetilde{T_s}'\right)\right)$ if $s_1$ is a minus.
\end{lemma}

\begin{proof}
When the numbers below and to the right of an empty box are equal, jeu-de-taquin promotion chooses to move the box below into the empty position. Since we construct the standard tableau $\widetilde{T_s}$ from $T_s$ by replacing the leftmost instance of a repeated entry with a smaller number than the rightmost instance, it follows that $\widetilde{jdt(T_s)} = jdt(\widetilde{T_s})$ when $s_1$ is a plus. If $s_1$ is a minus, then the numbers 1 and 2 must be promoted in the standard tableau to correspond to the promotion of a repeated entry of 1 in the semistandard tableau. Thus  $\widetilde{jdt(T_s)} = jdt\left(jdt\left(\widetilde{T_s}\right)\right)$ in the case that $s_2$ is a minus. The lemma follows from the fact that jeu-de-taquin promotion commutes with the process of conjugation in standard tableaux.
\end{proof}

We have been considering webs with boundary lying on a horizontal line, but  webs are often also viewed in a disk with univalent vertices on the boundary circle \cite{K,PPR}. Boundary vertices are enumerated counterclockwise with respect to some base point. To obtain a web with linear boundary, the circle bounding the disk is split open at the base point. The notion of web rotation is more natural when viewed from the disk perspective. Figure \ref{Rotate} has an example of rotation on a web with linear boundary.

\begin{theorem}
Jeu-de-taquin promotion of semistandard tableaux corresponds to rotation of webs.
\end{theorem}
\begin{proof}
Jeu-de-taquin promotion in standard tableaux corresponds to rotation of webs  with boundary $3n$ sources. Since $\widetilde{T_s}'$ is a standard tableau, the web $w_{jdt(\widetilde{T_s}')}$ is a rotation of the web $w_{\widetilde{T_s}'}$. It follows from Lemma \ref{semijdt} that the web $w_{\widetilde{jdt(T_s)}'}$ is also a rotation of $w_{\widetilde{T_s}'}$. Therefore the web obtained by rotation of $w_{T_s}$ is the same as the web $w_{jdt(T_s)}$. 
\end{proof}

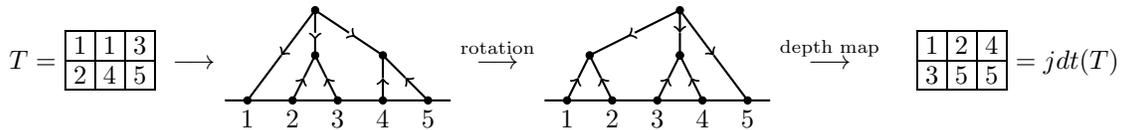
\begin{figure}[h]
\raisebox{12pt}{$T = $} \raisebox{4pt}{\young(113,245) } \raisebox{12pt}{$\longrightarrow$} \begin{tikzpicture}[baseline=0cm, scale=.6]
\draw[style=thick] (3.5,0)--(8.5,0);

\draw[radius=.08, fill=black](4,0)circle;
\draw[radius=.08, fill=black](5,0)circle;
\draw[radius=.08, fill=black](6,0)circle;
\draw[radius=.08, fill=black](7,0)circle;
\draw[radius=.08, fill=black](8,0)circle;

\draw[style=thick, ->](5.5,2) -- (4.75,1);
\draw[style=thick](4.75,1)--(4,0);

\draw[style=thick,->](7, 0) -- (7,.5);
\draw[style=thick](7,.5)--(7,1);
\draw[radius=.08, fill=black](7,1)circle;
\draw[style=thick,->](8,0)--(7.5,.5);
\draw[style=thick](7.5,.5)--(7,1);
\draw[radius=.08, fill=black](5.5,1)circle;
\draw[style=thick,-<](5.5,1)--(5.5,1.5);
\draw[style=thick](5.5,1.5)--(5.5,2);
\draw[radius=.08, fill=black](5.5,2)circle;
\draw[style=thick,->](5,0)--(5.25,.5);
\draw[style=thick](5.25,.5)--(5.5,1);
\draw[style=thick,->](6,0)--(5.75,.5);
\draw[style=thick](5.75,.5)--(5.5,1);
\draw[style=thick,-<](7,1)--(6.25,1.5);
\draw[style=thick](6.25,1.5)--(5.5,2);

\node at (4,-.4) {1};
\node at (5,-.4) {2};
\node at (6,-.4) {3};
\node at (7,-.4) {4};
\node at (8,-.4) {5};
\end{tikzpicture}
\raisebox{12pt}{$\stackrel{\textup{rotation}}{\longrightarrow}$}
\begin{tikzpicture}[baseline=0cm, scale=.6]
\draw[style=thick] (4.5,0)--(9.5,0);

\draw[radius=.08, fill=black](5,0)circle;
\draw[radius=.08, fill=black](6,0)circle;
\draw[radius=.08, fill=black](7,0)circle;
\draw[radius=.08, fill=black](8,0)circle;
\draw[radius=.08, fill=black](9,0)circle;

\draw[radius=.08, fill=black](7.5,1)circle;

\draw[radius=.08, fill=black](5.5,1)circle;
\draw[style=thick,-<](5.5,1)--(6.5,1.5);
\draw[style=thick](6.5,1.5)--(7.5,2);
\draw[radius=.08, fill=black](7.5,2)circle;
\draw[style=thick,->](5,0)--(5.25,.5);
\draw[style=thick](5.25,.5)--(5.5,1);
\draw[style=thick,->](6,0)--(5.75,.5);
\draw[style=thick](5.75,.5)--(5.5,1);
\draw[style=thick, ->](7.5,2) -- (7.5,1.5);
\draw[style=thick](7.5,1.5)--(7.5,1);
\draw[style=thick, ->](7.5,2)--(8.25,1);
\draw[style=thick](8.25,1)--(9,0);

\draw[radius=.08, fill=black](7.5,1)circle;

\draw[style=thick,->](7,0)--(7.25,.5);
\draw[style=thick](7.25,.5)--(7.5,1);
\draw[style=thick,->](8,0)--(7.75,.5);
\draw[style=thick](7.75,.5)--(7.5,1);
\node at (5,-.4) {1};
\node at (6,-.4) {2};
\node at (7,-.4) {3};
\node at (8,-.4) {4};
\node at (9,-.4) {5};
\end{tikzpicture}
\raisebox{12pt}{$\stackrel{\textup{depth map}}{\longrightarrow}$} \hspace{.15in}\raisebox{4pt}{\young(124,355) }\raisebox{12pt}{$= jdt(T)$}
\caption{Rotation of a web.} \label{Rotate}
\end{figure}

\subsection{Join and shuffling}
Tymoczko defines the notion of a shuffle of two standard tableaux and proves that the web associated to the shuffle is the join of the webs corresponding to those tableaux~\cite{T}. The join of two webs is the insertion of one into the other between some designated pair of vertices. The definition of shuffle has a natural extension to semistandard tableaux. 

Let $T$ be a semistandard tableau filled with at least one of each of the numbers $1,\ldots, \ell_1$ and $T'$ be a semistandard tableau filled with at least one of each of the numbers $1,\ldots, \ell_2$. Let $i\leq \ell_1$. The shuffle of $T'$ into $T$ at $i$ is denoted by $T'\stackrel{i}{\mapsto} T$ and defined by the following process. 
\begin{itemize}
\item{For each instance of $j = 1, \ldots, i$ in $T$, put $j$ in the same column of $T'\stackrel{i}{\mapsto} T$ as in $T$.}
\item{For each instance of $j = 1, \ldots, \ell_2$ in $T'$, put $j+i$ in the same column of $T'\stackrel{i}{\mapsto} T$ as in $T'$.}
\item{For each instance of $j=i+1, \ldots, \ell_1$ in $T$, put $j+\ell_2$ in the same column of $T'\stackrel{i}{\mapsto} T$ as in $T$.}
\end{itemize}
Figure \ref{joinweb} compares the shuffling of tableaux to the join of webs. The entries of $T'$ that have been shuffled into $T$ appear in bold.
\begin{figure}[h]
\raisebox{10pt}{$T =$} \raisebox{6pt}{ \young(114,235)} \raisebox{10pt}{$ \longrightarrow $}
\begin{tikzpicture}[baseline=0cm, scale=.6]
\draw[style=thick] (3.5,0)--(8.5,0);

\draw[radius=.08, fill=black](4,0)circle;
\draw[radius=.08, fill=black](5,0)circle;
\draw[radius=.08, fill=black](6,0)circle;
\draw[radius=.08, fill=black](7,0)circle;
\draw[radius=.08, fill=black](8,0)circle;

\draw[style=thick, ->] (8,0) arc (0:90: 2cm);
\draw[style=thick](6,2) arc (90:180:2cm);

\draw[xshift=2cm, style=thick,->](4, 0) -- (4,.25);
\draw[xshift=2cm, style=thick](4,.2)--(4,.5);
\draw[xshift=2cm, radius=.08, fill=black](4,.5)circle;
\draw[xshift=2cm, style=thick,->](3,0)--(3.5,.25);
\draw[xshift=2cm, style=thick](3.5,.25)--(4,.5);
\draw[xshift=2cm, style=thick,->](5,0)--(4.5,.25);
\draw[xshift=2cm,  style=thick](4.5,.25)--(4,.5);
\end{tikzpicture}
\hspace{.5in}
\raisebox{10pt}{$T' =$} \raisebox{6pt}{ \young(123)} \raisebox{10pt}{$ \longrightarrow $}
\begin{tikzpicture}[baseline=0cm, scale=.6]
\draw[ style=thick] (2.5,0)--(5.5,0);
\draw[ radius=.08, fill=black](3,0)circle;
\draw[ radius=.08, fill=black](4,0)circle;
\draw[radius=.08, fill=black](5,0)circle;

\draw[style=thick,->](4, 0) -- (4,.5);
\draw[style=thick](4,.5)--(4,1);
\draw[radius=.08, fill=black](4,1)circle;
\draw[style=thick,->](3,0)--(3.5,.5);
\draw[style=thick](3.5,.5)--(4,1);
\draw[style=thick,->](5,0)--(4.5,.5);
\draw[style=thick](4.5,.5)--(4,1);

\end{tikzpicture}

\vspace{.2in}

\raisebox{20pt}{$T' \stackrel{1}{\mapsto} T =$} \raisebox{7pt}{\young(11\blda,\bldb\bldc7,568)} \raisebox{20pt}{$ \longrightarrow $}
\begin{tikzpicture}[baseline=0cm, scale=.5]
\draw[style=thick] (3.5,0)--(11.5,0);

\draw[radius=.08, fill=black](4,0)circle;
\draw[radius=.08, fill=black](5,0)circle;
\draw[radius=.08, fill=black](6,0)circle;
\draw[radius=.08, fill=black](7,0)circle;
\draw[radius=.08, fill=black](8,0)circle;
\draw[radius=.08, fill=black](9,0)circle;
\draw[radius=.08, fill=black](10,0)circle;
\draw[radius=.08, fill=black](11,0)circle;

\draw[style=thick, ->] (11,0) arc (0:90: 3.5cm);
\draw[style=thick](7.5,3.5) arc (90:180:3.5cm);

\draw[xshift=2cm, style=thick,->](4, 0) -- (4,.5);
\draw[xshift=2cm, style=thick](4,.5)--(4,1);
\draw[xshift=2cm, radius=.08, fill=black](4,1)circle;
\draw[xshift=2cm, style=thick,->](3,0)--(3.5,.5);
\draw[xshift=2cm, style=thick](3.5,.5)--(4,1);
\draw[xshift=2cm, style=thick,->](5,0)--(4.5,.5);
\draw[xshift=2cm, style=thick](4.5,.5)--(4,1);
\draw[xshift=5cm, style=thick,->](4, 0) -- (4,.5);
\draw[xshift=5cm,style=thick](4,.5)--(4,1);
\draw[xshift=5cm,radius=.08, fill=black](4,1)circle;
\draw[xshift=5cm,style=thick,->](3,0)--(3.5,.5);
\draw[xshift=5cm,style=thick](3.5,.5)--(4,1);
\draw[xshift=5cm,style=thick,->](5,0)--(4.5,.5);
\draw[xshift=5cm,style=thick](4.5,.5)--(4,1);
\end{tikzpicture}
\caption{Joining two webs.}\label{joinweb}
\end{figure}
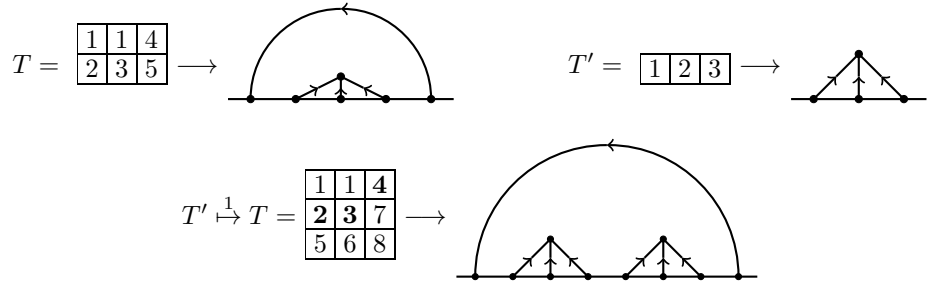

In the case where $T$ has content $s = s_1 \cdots s_{\ell_{1}}$ and $T'$ has content $t = t_1 \cdots t_{\ell_{2}}$, the web $T'\stackrel{i}{\mapsto} T$ will be semistandard with content $s_1 \cdots s_i t_1 \cdots t_{\ell_{2}} s_{i+1} \cdots s_{\ell_{1}}$. The following theorem, illustrated in Figure \ref{joinweb}, is a straightforward generalization of Tymoczko's result~\cite{T}.
\begin{theorem}
Shuffling of semistandard tableaux corresponds to the join of webs. Specifically, shuffling $T'$ into $T$ at $i$ corresponds to inserting the web $w_{T'}$ into the web $w_T$ between vertices $i$ and $i+1$ of $T$.
\end{theorem}

\bibliographystyle{plain}
\bibliography{SemiStd}

\end{document}